\documentclass[10pt, a4paper]{amsart} 
\usepackage{color}
\usepackage{amscd,amssymb,graphics}
\usepackage{amsfonts}
\usepackage{amsmath}
\usepackage{tikz-cd}
\usepackage{multirow}
\usepackage{xparse}
\usepackage[createShortEnv, conf={no link to theorem}]{proof-at-the-end}
\input xy
\xyoption{all}
\usepackage{epsfig}

\newtheorem{thm}{Theorem}[section]

\newtheorem{f}[thm]{Fact}
\newtheorem{cor}[thm]{Corollary}

\newtheorem{lemma}[thm]{Lemma}
\newtheorem{prop}[thm]{Proposition}

\newtheorem{q}[thm]{Question}

\theoremstyle{definition}
\newtheorem{defin}[thm]{Definition}

\theoremstyle{remark}
\newtheorem{remark}[thm]{Remark}

\newtheorem{ex}[thm]{Example}

\numberwithin{equation}{section}

\newcommand{\delete}[1]{} 

\def\eps{{\varepsilon}}
\newcommand{\sk}{\vskip 0.1cm}

\newcommand{\neighbor}[2][ ]{\mathcal{N}_{#1}(#2)}
\newcommand{\ben}{\begin{enumerate}}

\newcommand{\een}{\end{enumerate}}
\newcommand{\bit}{\begin{itemize}}

\newcommand{\eit}{\end{itemize}}

\newcommand{\absmall}[2][\mathfrak{B}]{\mathrm{small}_{a, r}\left({#1}, {#2}\right)}
\newcommand{\fbsmall}[2][\mathfrak{B}]{\mathrm{small}_{g, r}\left({#1}, {#2}\right)}

\def\A {{\mathcal A}}
\def\R {{\mathbb R}}
\def\N {{\mathbb N}}
\def\Z {{\mathbb Z}}

\newcommand{\RUC}{\mathrm{RUC}_{b}}
\newcommand{\LUC}{\mathrm{LUC}_{b}}
\newcommand{\ruc}{\mathrm{ruc}}
\newcommand{\luc}{\mathrm{luc}}

\def\sign{{\mathrm sign}\,}

\def\mol{\mathrm{Mol}}

\def\A{{\mathcal{A}}}

\def\QED{\nobreak\quad\ifmmode\roman{Q.E.D.}\else{\rm Q.E.D.}\fi}

\def\BTame{\mathrm{\mathbf{[T]}}}
\def\BNP{\mathrm{\mathbf{[ASP]}}}
\def\BDLP{\mathrm{\mathbf{[WRC]}}}

\def\spn{\operatorname{Span}}

\def\tame{\operatorname{Tame}}
\def\tamefunctional{\operatorname{\mathit{Tame}}}
\def\wap{\operatorname{WAP}}
\def\wapfunctional{\operatorname{\mathit{WAP}}}
\def\asp{\operatorname{Asp}}
\def\aspfunctional{\operatorname{\mathit{Asp}}}

\newcommand{\acx}{{\rm{acx\,}}}
\newcommand{\bo}{{\rm{bal\,}}}

\graphicspath{ {media/} }

\usepackage[hidelinks]{hyperref}
\begin{document}
	
\numberwithin{equation}{thm}
\title[]
{The Banach Algebra $L^{1}(G)$ and Tame Functionals} 

\sk

\author[]{Matan Komisarchik}
\address{Department of Mathematics,
	Bar-Ilan University, 52900 Ramat-Gan, Israel}
	\email{matan.komisarchik@biu.ac.il}


\date{2025, March}

\begin{abstract}
	We give an affirmative answer to a question due to M. Megrelishvili, and show that for every locally compact group $G$ we have $\tamefunctional(L^{1}(G)) = \tame(G)$, which means that a functional is tame over $L^{1}(G)$ if and only if it is tame as a function over $G$.
	In fact, it is proven that for every norm-saturated, convex vector bornology on $\RUC(G)$, being small as a function and as a functional is the same.
	This proves that $\aspfunctional(L^{1}(G)) = \asp(G)$ and reaffirms a well-known, similar result which states that $\wap(G) = \wapfunctional(L^{1}(G))$. 
\end{abstract}




\maketitle
\setcounter{tocdepth}{1}
\tableofcontents


\section{Introduction}
Let $G$ be a locally compact group and let $\varphi \in \RUC(G)$ be a bounded, right uniformly continuous function.
We can also consider the same $\varphi$ as a functional $\varphi \in L^{\infty}(G)$ over the Banach algebra $L^{1}(G)$.
Thus, the properties of $\varphi$ can be explored in two possibly distinct ways.
Recall the following definitions of weakly almost periodic (WAP) functions/functionals:
\ben
	\item $\varphi$ is \emph{weakly almost periodic as a function} if $\varphi \cdot G \subseteq \RUC(G)$ is weakly relatively compact.
	In this case we write $\varphi \in \wap(G)$.
	\item $\varphi$ is \emph{weakly almost periodic as a functional} if the operator $\mathfrak{L}_{\varphi}\colon L^{1}(G) \to L^{\infty}(G)$ defined by $a \mapsto \varphi \circledast a$ is weakly compact. Here $\circledast$ is the induced Banach module action satisfying 
	$$
	  \langle \varphi \circledast a, b \rangle = \langle \varphi, a \star  b\rangle
	$$ 
	for every $b \in L^{1}(G)$.
	In this case we write $\varphi \in \wapfunctional(L^{1}(G))$.
\een
WAP functions and functionals have been studied extensively.
Some recent developments can be found in \cite{Filali} and its references.

Please note that both definitions of $\wap(G)$ and $\wapfunctional(L^{1}(G))$ are similar in the sense that they pose the same ``size" restriction (e.g., weak relative compactness) on different subsets of $L^{1}(G)$: $\varphi \cdot G$ for functions and $\mathfrak{L}_{\varphi}(B_{L^{1}(G)}) = \varphi \circledast B_{L^{1}(G)}$ for functionals.
In fact, it is known that ${\wap(G) = \wapfunctional(L^{1}(G))}$ \cite[Thm.~4]{Ulger}. 
It is therefore very natural to investigate the relationship between the analogous definitions for different size restrictions.

One prominent approach from the theory of dynamical systems gives rise to tame and Asplund functions and functionals. 
These notions generalize weak almost periodicity and have been of great importance in the research of function dynamics over groups and semi-groups (see, for example, \cite{GM-MTame} or \cite{MegrelBook}). 
Therefore, they can be expected to be useful in this context as well.
Indeed, in \cite{TameFunc}, Megrelishvili used a very similar approach and asked whether tame/Asplund functions are the same as tame/Asplund functionals over locally compact groups.
He managed to prove this for discrete groups but the general case remained open.

In this paper, we answer Megrelishvili's question affirmatively.
Moreover, we show that it is actually true for many other ``smallness" criteria.
Formally, we can consider a vector bornology $\mathcal{B}$ on $\RUC(G)$ and define:
	
\ben
	\item $\varphi$ is \emph{right $\mathcal{B}$-small as a function} if $\varphi \cdot G \in \mathcal{B}$ (Definition \ref{defin:B_small_function}).
	In this case we write $\varphi \in \fbsmall[\mathcal{B}]{\RUC(G)}$.
	\item $\varphi$ is \emph{right $\mathcal{B}$-small as a functional} if $\varphi \circledast B_{L^{1}(G)} \in \mathcal{B}$ (Definition \ref{defin:B_small_functional}).
	In this case we write $\varphi \in \absmall[\mathcal{B}]{\RUC(G)}$.
\een
A detailed study of tame and Asplund subsets as norm-saturated, convex vector bornologies, can be found in \cite{TameLCS}.

\newtheorem*{thm:functionfunctionalequivalence}{Theorem \ref{thm:function_functional_equivalence}}
\begin{thm:functionfunctionalequivalence}
	For every norm-saturated, convex vector bornology $\mathcal{B}$ on $\RUC(G)$ we have:
	$$
	\varphi\cdot G \in \mathcal{B} \iff \mathfrak{L}_{\varphi}(B_{\RUC(G)}) = \varphi \circledast B_{\RUC(G)} \in \mathcal{B}.
	$$
	In particular:
	\begin{align*}
		\tamefunctional(L^{1}(G)) & = \tame(G)\\
		\aspfunctional(L^{1}(G)) & = \asp(G)\\
		\wapfunctional(L^{1}(G)) & = \wap(G).
	\end{align*}
\end{thm:functionfunctionalequivalence}
\vskip 0.5cm
The paper is structured in the following way:
\begin{itemize}
	\item Sections \ref{section:preliminaries} and \ref{section:group_algebra_actions} contains preliminary material and notations.
	\item Section \ref{section:tameness_and_cotameness} is a short study of the relation between tameness and co-tameness in the dual of a Banach space.
	In fact, we show a new method for constructing independent sequences in the dual, proving that every tame weak-star compact subset of $V^{*}$ is also co-tame (Theorem~\ref{thm:tame_inversion}).
	\item Section \ref{section:small_functions_and_functionals} defines the general notion of small functions and functionals over groups and Banach algebras. 
	We explore some properties of these definitions and also show a few examples.
	Mainly, we present a classification of tame coordinate functionals on the Banach algebra $\mathcal{L}(l^{1}, l^{1})$ using the method from the previous section.
	\item In Section \ref{section:main}, we explore the relations between tame functions and functionals on $\RUC(G)$ for a locally compact $G$ and prove Theorem \ref{thm:function_functional_equivalence}.
	\item Finally, in Section \ref{section:open_questions} we present some open questions.
\end{itemize}

\section*{Acknowledgments}
The author is deeply thankful to Jan Pachl, for his very useful suggestion to apply the UEB topology \cite{pachl2012uniform}.
This new idea allowed for a complete proof of Theorem \ref{thm:function_functional_equivalence} and an affirmative answer to Megrelishvili's question.
The author would also like to express his gratitude for the opportunity to share these results for the first time at the conference Algebra, Topology and Their Interactions, 2024.
Finally, he would like to thank the referee for his detailed and thoughtful comments.

\section{Preliminaries} \label{section:preliminaries}
	All vector spaces in this paper are considered over the field $\R$ of real numbers.
	Also, we assume that all topological spaces are Tychonoff and all topological groups are locally compact.
	If $X$ is a topological space, let $C(X)$ and $C_{b}(X)$ be the spaces of continuous, and bounded continuous functions over $X$, respectively. 
	A subset $A$ of a Banach space $V$ is weakly relatively compact if its closure is weakly compact.
	Write $B_{V}$ for the unit ball of $V$.
	We will use the definitions and notations of \cite{TameLCS}.
	\subsection{Fragmented Families}
	Fragmentability, and fragmented functions had been studied in a variety of settings (multi-valued functions, uniform spaces, etc.).
	Here we show a specific (but equivalent) version which is most relevant.
	More general approaches can be found in \cite{JR,JOPV,Me-fr98}.
	\begin{defin} \label{def:fr} \cite{JOPV,Me-fr98} 
		Let $(X,\tau)$ be a topological space.
		A (not necessarily continuous) function $f\colon X \to \R$ is said to be \emph{fragmented} if for every nonempty subset $A$ of $X$ and every $\eps > 0$ there exists an open subset $O$ of $X$ such that $O \cap
		A$ is nonempty and the set $f(O \cap A)$ is $\eps$-small in $\R$.
	\end{defin}

	\begin{defin}\cite{GM1} \label{d:fr-family} \ 
		We say that a {\it family of functions} $\mathcal{F} \subseteq \R^{X}$ is {\it fragmented} if the condition of Definition
		\ref{def:fr} holds simultaneously for all $f \in \mathcal F$.
		That is, $f(O \cap A)$ is $\eps$-small for every $f \in \mathcal
		F$. 
	\end{defin} 
	\begin{defin}\label{d:AspSet} \cite[p.~22]{Fabian1997} 
		Let $A$ be a bounded subset of a Banach space $V$.
		$A$ is said to be an \emph{Asplund set} if for every countable $C \subseteq A$ the pseudometric space $(B_{V^{*}}, \rho_{C})$ is separable, where:
		$$
		  \rho_{C}(\varphi, \psi) := \sup_{x \in C} \lvert \varphi(x) - \psi(x) \rvert.
		$$
	\end{defin}
	\begin{f}
		\label{f:countDetermined} \cite[Theorem 2.6]{GM-MTame}
		Let $F$ be a
		bounded
		family of 
		real-valued continuous functions on a compact 
		space $X$.
		The following conditions are equivalent:
		\ben
		\item
		$F$ is a fragmented family of functions on $X$.
		\item
		Every \emph{countable} subfamily $C$ of $F$ is fragmented on $X$.
		\item
		For every countable subfamily $C$ of $F$,
		the pseudometric space $(X,\rho_{C})$ is separable,
		where
		$$
		\rho_{C}(x_1,x_2):=\sup_{f \in C} |f(x_1)-f(x_2)|.
		$$ 
		\een
	\end{f}
	\begin{cor}\label{cor:countDetermined}
		A bounded subset of a Banach space $V$ is Asplund if and only if it is fragmented as a family of functions over $B_{V^{*}}$ furnished with the weak-star topology.
	\end{cor}
	Recall that a Banach space $V$ is \emph{Asplund} if the dual of every separable closed subspace is separable.
	\vskip 0.01cm
	\begin{f} \label{f:Asp-charact} \rm{(Namioka--Phelps \cite{NP})} 
		A Banach space $(V,||\cdot||)$ is Asplund if and only if every bounded weak-star compact subset $K \subset E^*$ is 
		(weak$^*$,norm)-fragmented.  
	\end{f}
	\subsection{Tame Families}
	A sequence of real functions  ${\{f_n\colon X \to \R\}_{n \in \N}}$  on a set
	$X$ is said to be 
	(combinatorially)  
	\emph{independent} (see \cite{Ros0,Tal}) if
	there exist real numbers $a < b$ (\emph{bounds of independence}) such that
	$$
	\bigcap_{n \in P} f_n^{-1}\left((-\infty,a]\right) \cap  \bigcap_{n \in M} f_n^{-1}\left([b,\infty)\right) \neq \emptyset
	$$
	for all finite disjoint subsets $P, M$ of $\N$. 
	
	\begin{defin} \label{d:TameFamily} \cite{GM-tame,GM-MTame}
		A bounded family $F$ of 
		real-valued (not necessarily continuous)  functions on a set $X$ is a {\it tame family} if $F$ does not contain an independent sequence. 
	\end{defin}

	The following fact
	can be easily derived using the finite intersection property characterization of compactness (see, for example, \cite{Ko}). 
	\begin{f} \label{f:independence_over_compact_sets} 
		Suppose that $\{f_{n}\}_{n \in \N}$ is 
		an independent 
		family of continuous functions over a compact $X$.
		Then there are $a < b \in \R$ such that for every disjoint, possibly infinite $P, M \subseteq \N$:
		$$
		\bigcap_{n \in P} f_n^{-1}\left((-\infty,a]\right) \cap  \bigcap_{n \in M} f_n^{-1}\left([b,\infty)\right) \neq \emptyset.
		$$
	\end{f}
	\begin{f}\cite[Lemma~6.4.3]{GM-MTame}
		\label{fact:tame_over_closure}
		Let $\{f_{n}\}_{n \in \N}$ be a bounded sequence of continuous functions on a topological space $X$. 
		Let $Y$ be a dense subset of $X$.
		Then $\{f_{n}\}_{n \in \N}$ is an independent sequence on $X$ if and only if the sequence of restrictions $\{f_{n}\mid_{Y}\}_{n \in \N}$ is an independent sequence on $Y$.
	\end{f}
	As in \cite{GM-rose,GM-fpt,GM-survey}, we say that a Banach space $V$ is \textit{Rosenthal} if any bounded sequence contains a weak Cauchy subsequence,
	or equivalently, if $V$ does not contain an isomorphic copy of $l^1$. 
	Such Banach spaces appear in many publications (especially, after Rosenthal's classical work \cite{Ros0}),
	usually without any special name. 
	\begin{defin}
		Let $V$ be a Banach space:
		\ben
			\item A bounded subset $A \subseteq V$ is said to be \emph{tame} if it is tame as a family of functions over $B_{V^{*}}$.
			\item A weak-star compact subset $M \subseteq V^{*}$ is said to be \emph{co-tame} if every bounded $A \subseteq V$ is tame as a family of functions over $M$.
		\een
	\end{defin}

	The following fact is a nontrivial reformulation of some known results of Rosenthal.
	\begin{f} \cite[Example~2.5.1]{GM-MTame}
		Let $V$ be a Banach space, then the following conditions are equivalent:
		\ben
			\item $V$ is a Rosenthal Banach space.
			\item The unit ball $B_{V}$ of $V$ is tame as a family of functions over the weak-star compact unit ball $B_{V^{*}}$.
			In other words, $B_{V}$ is tame as a subset of $V$.
		\een
	\end{f}
	The following is a consequence of \cite[Prop.~4.19]{GM-rose} and Fact \ref{fact:tame_over_closure}.
	\begin{prop} \label{fact:tame_as_banach_subset_equivalence}
		A bounded subset $A \subseteq C_{b}(X)$ is tame over $X$ if and only if it is tame as a subset of $C_{b}(X)$.
	\end{prop}
	\begin{proof}
		Write $V := C_{b}(X)$ and recall that we only consider Tychonoff spaces.
		It is well known that $X$ is embedded in $B_{V^{*}}$ via the natural map $J\colon X \to B_{V^{*}}$ defined by $(J(x))(f) := f(x)$.
		Thus, if $A \subseteq C_{b}(X)$ is tame as a subset of $V$, it is necessarily tame over $X$.
		
		To see the converse, consider the Stone-\u{C}ech compactification $i_{X}\colon X \to \beta X$. Since $X$ is Tychonoff, $i_{X}$ is a topological embedding \cite[p.~152]{kelley}.
		Write $\rho\colon C(\beta X) \to C_{b}(X)$ for the restriction map defined by $\rho(f) := f \circ i_{X}$.
		Also, for every $f \in C_{b}(X)$, write $\beta f \in C(\beta X)$ for its unique extension.
		Finally, let $W := C(\beta X)$.
		
		By contradiction, assume that $A$ is tame over $X$, but not tame as a subset of $V$.
		Thus, we can find $\{f_{n}\}_{n \in \N} \subseteq A$ which is independent over $B_{V^{*}}$.
		In virtue of Fact \ref{fact:tame_over_closure}, $\beta A := \{\beta f \mid f \in A\}$ is tame over $\beta X$.
		Using \cite[Prop.~4.19]{GM-rose}, we conclude that $\beta A$ is tame as a subset of $W$.
		Now, consider $\Phi\colon V^{*} \to W^{*}$ defined as 
		$$
		  \forall \varphi \in V^{*}, g \in W:  (\Phi(\varphi))(g) := \varphi(\rho(g)).
		$$
		It is easy to see that $\Phi(B_{V^{*}}) \subseteq B_{W^{*}}$.
		Also for every $\varphi \in V^{*}$ and $f \in V$ we have:
		$$
		  (\Phi(\varphi))(\beta f) := \varphi(\rho(\beta f)) = \varphi(f).
		$$
		As a consequence, if $\{f_{n}\}_{n \in \N}$ is independent over $B_{V^{*}}$, then $\{\beta f_{n} \}_{n \in \N}$ is independent over ${\Phi(B_{V^{*}}) \subseteq B_{W^{*}}}$.
		However, this is impossible since $\{\beta f_{n} \}_{n \in \N} \subseteq \beta A$ which is tame over $B_{W^{*}}$.
	\end{proof}

	\subsection{Topological Groups} \label{subsection:topological_groups}
	\begin{defin} \cite[Definition~4.11]{AbstractHarmonic}
		Let $G$ be a topological group and $U \subseteq G$ a neighborhood of $e \in G$.
		Let us write:
		$$
		U_{L} := \{(x, y) \in G \times G \mid x^{-1}y \in U\},
		$$
		$$
		U_{R} := \{(x, y) \in G \times G \mid xy^{-1} \in U\}.
		$$	
		$\{U_{L}\}_{U \in \neighbor{e}}$ and $\{U_{R}\}_{U \in \neighbor{e}}$ are bases of the left ($\mu_{G}^{l}$) and right ($\mu_{G}^{r}$) uniformities, respectively.
	\end{defin}
	
	\subsection{Group Actions} \label{subsection:group_actions}
	Let $G$ be a group acting from the left (resp. right) on a set $X$.
	For every $x \in X$, write $O_{x}\colon G \to X$ for the orbit map defined by 
	$$
	O_{x}(g) := g \cdot x \text{ (resp. } O_{x}(g) := x \cdot g \text{ )}.
	$$
	Also, there is a natural induced \emph{right (resp. left)} action on $C_{b}(X)$ defined as:
	$$
	(\varphi \cdot g)(x) := \varphi(g \cdot x) \text{ (resp. } (g \cdot \varphi)(x) := \varphi(x \cdot g) \text{ )},
	$$
	for every $\varphi \in C_{b}(X)$.
	Any group has a natural left and right action on itself.
	In this case, we will write $O_{g}^{l}$/$O_{g}^{r}$ for the left/right orbit maps, respectively.
	
	\begin{defin}\cite[p.~275]{AbstractHarmonic} 
		Let $G$ be a topological group.
		A function $\varphi\in C_{b}(G)$ is said to be \emph{right (resp. left) uniformly continuous} if $\varphi$ is uniformly continuous with respect to $\mu_{G}^{r}$ (resp. $\mu_{G}^{l}$).
		In this case, we write $\varphi \in \RUC(G)$ (resp. $\varphi \in \LUC(G)$).
	\end{defin}

	\begin{defin}\cite{Ulger}\label{defin:ulger_wap_function}
		A bounded function $\varphi \in C_{b}(G)$ is said to be \emph{weakly almost periodic (WAP)} if $\varphi \cdot G$ is weakly relatively compact.	
		We will write $\wap(G)$ for the family of all WAP functions over $G$.
	\end{defin}
	It is known that every WAP function is also right uniformly continuous (see for example \cite[Thm.~2]{Ulger}), so we can always assume that $\varphi \in \RUC(G)$.
	
	\begin{defin}\cite[Def.~3.1]{GM-MTame}
		A function $\varphi \in \RUC(G)$ is said to be \emph{tame} if $\varphi \cdot G$ is tame over $G$.
		
		We will write $\tame(G)$ for the family of all tame functions over $G$.
	\end{defin}

	The case of Asplund functions needs to be stated more carefully.
	Suppose that $K$ is compact, $D \subseteq K$ is dense and $F \subseteq C(K)$ is a bounded family.
	As we have established, $F$ is tame over $D$ iff $F$ is tame over $K$ iff $F$ is tame over $B_{C(K)^{*}}$.
	In their original definition (\ref{d:AspSet}), Asplund sets are only defined inside Banach spaces.
	Using Fact \ref{f:countDetermined}, we can equivalently require $F$ to be fragmented over $K$.
	However, this approach is only valid for compact spaces (as far as the author is aware).
	In order to use this technique for general groups, we will use $G$-compactifications.
	
	Suppose that $G$ is a topological group that acts from the left on the topological spaces $X, Y$.
	It is said that $Y$ is a \emph{$G$-compactification} of $X$, if $Y$ is compact, there exists a continuous $\rho\colon X \to Y$ such that $\rho(X)$ is dense in $Y$, and:
	$$
	  \forall x \in X, g \in G: \rho(g \cdot x) = g \cdot \rho(x).
	$$
	
	\begin{defin} \label{defin:asplund_function}
		A function $\varphi \in \RUC(G)$ is \emph{Asplund} if one of the following equivalent conditions hold (Corollary \ref{cor:countDetermined}):		
		\ben
		\item $\varphi \cdot G$ is an Asplund subset of $\RUC(G)$ \cite[Def.~7.9]{Me-Frag04}.
		\item $\varphi \cdot G$ is fragmented over every $G$-compactification of $G$ \cite[Def.~7.1]{GM1}.
		\een
		
		We will write $\asp(G)$ for the family of all Asplund functions over $G$.
	\end{defin}

	\subsection{Vector Bornologies}
	\begin{defin}\cite[Def.~1:1'1-3]{bornologies}
		A \emph{bornology} $\mathcal{B}$ on a vector space $V$ is a family of subsets in $V$ which covers $V$, and is hereditary 
		under inclusion (i.e., if $A \in \mathcal{B}$ and $B \subseteq A$ then $B \in \mathcal{B}$) and finite unions.
		
		We will say that $\mathcal{B}$ is a \emph{vector bornology} if for every $A, B \in \mathcal{B}$, $\alpha \in \R$:
		\ben
			\item $A + \alpha B \in \mathcal{B}$.
			\item $\bo A := \{ \beta x \mid x\in A, \beta \in [-1, 1]\} \in \mathcal{B}$.
		\een
		Write $\acx A$ for the \emph{absolutely convex hull of $A$}, meaning that:
		$$
		  \acx A := \left\{ \sum_{n=1}^{N} \alpha_{n} x_{n} \mid \alpha_{n} \in [-1, 1], x_{n} \in A, \sum_{n=1}^{N} \left\lvert\alpha_{n}\right\rvert \leq 1  \right\}.
		$$
		If $\mathcal{B}$ is also closed for convex hulls, it is said to be a \emph{convex vector bornology}.
		A bornology $\mathcal{B}$ is said to be \emph{saturated} with respect to some topology $\tau$ if whenever $A \in \mathcal{B}$, then $\overline{A}^{\tau} \in \mathcal{B}$.
	\end{defin}
	\begin{f}\cite[Prop.~3.2 and Lemma~3.5]{TameLCS} \label{fact:bornologies}
		Let $V$ be a Banach space. 
		The weakly relatively compact, Asplund and tame subsets of $V$ are norm-saturated, convex bornologies over $V$.
		We will denote them by $\BDLP^{V}, \BNP^{V}$ and $\BTame^{V}$, respectively.
		In cases where no confusion can arise, we will simply write $\BDLP, \BNP$ and $\BTame$.
	\end{f}
	\begin{remark}
		In \cite{TameLCS}, the authors do not consider weakly relative compact subsets but rather DLP subsets of locally convex spaces.
		In virtue of \cite[Thm.~17.12]{KellyNamioka}, both notions are equivalent in the case of Banach spaces, so we will implicitly apply their results to $\BDLP$.
		Similarly, they do not consider Asplund sets directly, but rather NP subsets, which are also equivalent in Banach spaces.
	\end{remark}

	\subsection{Actions of Banach Algebras on Their Dual}
	\label{subsection:banach_algebra_action_on_dual}
	Suppose that $\A$ is a Banach algebra.
	As in \cite{Arens}, the dual $\A^{*}$ can be furnished with an $\A$-bimodule structure via:
	\begin{align*}
		\forall a, b \in \A, \varphi \in \A^{*}: (a \circledast \varphi)(b) & := \varphi(ba)\\
		\forall a, b \in \A, \varphi \in \A^{*}: (\varphi \circledast a)(b) & := \varphi(ab).
	\end{align*}
	The use of a the $\circledast$ notation will be convenient later to distinguish between multiple types of actions in the case of the group algebra.

	\begin{defin} \cite[Section~3]{LauFunctionals}
		Let $\A$ be a Banach algebra.
		Then
		$$
		  \luc(\A) := \overline{\spn}(\A \circledast \A^{*}),
		$$
		$$
		  \ruc(\A) := \overline{\spn}(\A^{*} \circledast \A).
		$$
	\end{defin}
	Note that in \cite[p.~73]{Yassin}, the definitions are reversed.
	
	\begin{f} \cite[p.~198]{LauFunctionals}
		If $G$ is a locally compact group, then:
		$$
		  \LUC(G) = \luc(L^{1}(G)),\ \RUC(G) = \ruc(L^{1}(G)).
		$$
	\end{f}
	\begin{defin}\cite[Definition.~15.1]{ApproximatedIdentities}
		Let $\A$ be a Banach algebra and $M$ be an $\A$-module with the operation $\circledast$.
		$M$ is said to be \emph{essential} if $\A \circledast M = M$.
	\end{defin}

	\begin{f} \cite[Cor.~15.3]{ApproximatedIdentities} \label{fact:essential_is_module_identity}
		Let $\A$ be a Banach algebra with a bounded left approximate identity $\{e_{\lambda}\}_{\lambda \in \Lambda}$.
		Then a left Banach $\A$-module $M$ is essential if and only if for every $x \in M$ we have $\lim\limits_{\lambda \in \Lambda} e_{\lambda} \circledast x = x$.
	\end{f}
	\begin{f} \cite[Thm.~13.4]{ApproximatedIdentities} \label{fact:L1G_has_approximate_identity}
		The group algebra $L^{1}(G)$ has a two-sided approximate identity bounded by $1$.
	\end{f}
	
	\subsection{Classes of Functionals on Banach Algebras}\label{subsection:classes_of_functionals}
	Let $\A$ be a Banach algebra.
	Every functional $\varphi \in \A^{*}$ induces the linear maps $\mathfrak{L}_{\varphi}, \mathfrak{R}_{\varphi}\colon \A \to \A^{*}$ defined by:
	\begin{align*}
		\mathfrak{R}_{\varphi}(a) & := a \circledast \varphi,\\
		\mathfrak{L}_{\varphi}(a) & := \varphi \circledast a.
	\end{align*}
	\begin{defin}\cite[p.~198]{LauFunctionals}
		Let $\A$ be a Banach algebra.
		A functional $\varphi \in \A^{*}$ is said to be \emph{weakly almost periodic (WAP)} if the map $\mathfrak{L}_{\varphi}\colon \A\to\A^{*}$ is weakly compact, meaning that $\mathfrak{L}_{\varphi}(B_{\A}) = \varphi \circledast B_{\A}$ is weakly relatively compact in $\A^{*}$.
		
		We will write $\wapfunctional(\A)$ for the set of weakly almost periodic functionals over $\A$.
	\end{defin}
	
	\begin{defin}\cite[Def.~2.5]{TameFunc}
		Let $\A$ be a Banach algebra.
		A functional $\varphi \in \ruc(\A)$ is said to be \emph{tame/Asplund} if $\mathfrak{L}_{\varphi}$ factors through a Rosenthal/Asplund Banach space, meaning that there exists bounded operators $\alpha\colon \A \to V, \beta\colon V \to \A^{*}$ such that $\mathfrak{L}_{\varphi} = \beta \circ \alpha$ and $V$ is a Rosenthal/Asplund Banach space.
		
		We will write $\tamefunctional(\A)$/$\aspfunctional(\A)$ for the sets of tame and Asplund functionals, respectively.
	\end{defin}
	
	Recall that an operator $T\colon V\to W$ between two Banach spaces is weakly relatively compact/Asplund/tame if $T(B_{V})$ is weakly relatively compact/Asplund/tame in $W$, respectively.
	The following fact is a consequence of the Davis-Figiel-Johnson-Pe\l{}czy\'{n}ski factorization which can be found in \cite{DFJP}.
	Its application to reflexive, Asplund and Rosenthal spaces can be found in \cite[Cor.~3]{DFJP}, \cite[Sec.~1.3]{Fabian1997}, \cite[Thm.~6.3]{GM-rose}, respectively.
	\begin{f}
		Let $T\colon V \to W$ be an operator between Banach spaces.
		Then $T$ is weakly relatively compact/Asplund/tame if and only if it factors through a reflexive/Asplund/Rosenthal Banach space.
	\end{f}
	\begin{cor} \label{cor:equivalent_tame_functional_definition}
		Let $\A$ be a Banach algebra.
		A functional $\varphi \in \ruc(\A^{*})$ is Asplund/tame if and only if $\mathfrak{L}_{\varphi}(B_{\A})$ is Asplund/tame as a subset of $\A^{*}$.
	\end{cor}

	\subsection{Bounded Uniformly Equicontinuous Functions}
	\begin{defin}\cite{JanUEB}
		A subset $F\subseteq \RUC(G)$ (resp. $\LUC(G)$) is said to be \emph{right (resp. left) bounded uniformly equicontinuous (UEB)} if it is bounded and for every $\eps > 0$ there exists an entourage $\delta \in \mu_{G}^{r}$ (resp. $\mu_{G}^{l}$) such that:
		$$
		(g, h) \in \delta \Rightarrow \forall \varphi\in F: \lvert \varphi(g) - \varphi(h) \rvert < \eps.
		$$
	\end{defin}
	
	\begin{f}\cite[Lemma~4.2.1]{JanUEB} \label{fact:orbits_are_UEB}
		For a topological group $G$ we have:
		\begin{align*}
			F \subseteq \RUC(G) \text{ is right UEB} &\Rightarrow G\cdot F = \bigcup\{G\cdot \varphi \mid \varphi \in F\} \text{ is right UEB}\\
			F \subseteq \LUC(G) \text{ is left UEB} &\Rightarrow F\cdot G = \bigcup\{\varphi\cdot G \mid \varphi \in F \} \text{ is left UEB}.\\
		\end{align*}
		In particular:
		\begin{align*}
			\varphi \in \RUC(G) &\Rightarrow G\cdot \varphi \text{ is right UEB}\\
			\varphi \in \LUC(G) &\Rightarrow \varphi\cdot G \text{ is left UEB}.\\
		\end{align*}
	\end{f}
	\begin{remark}
		In \cite{JanUEB}, the authors use an opposite notation. 
		Namely, their LUC is our RUC and vice versa.
		We adapted their results to fit our notation.
	\end{remark}

	\begin{defin}\cite{JanUEB}
		The right UEB topology on $\RUC(G)^{*}$ is the topology of uniform convergence on right UEB subsets of $\RUC(G)$.
		Namely, a net ${\{\mu_{\lambda}\}_{\lambda \in \Lambda} \subseteq \RUC(G)^{*}}$ converges to $\mu \in \RUC(G)^{*}$ in the right UEB topology if and only if for every UEB subset $F \subseteq \RUC(G)$ and $\eps > 0$ there exists $\lambda_{0} \in \Lambda$ such that:
		$$
		\forall \lambda_{0} \leq \lambda \in \Lambda\ \forall \varphi \in F: \lvert \langle \varphi, \mu_{\lambda} - \mu \rangle\rvert < \eps.
		$$
		
		The left UEB topology on $\LUC(G)^{*}$ is defined analogously.
		
	\end{defin}

	
	Let $G$ be a locally compact group.
	A point mass is a measure $\rho$ of $G$ such that $\rho(A) = 1_{A}(g)$ for some $g \in G$.
	A measure on $G$ is said to be \emph{molecular} if it a linear combination of point masses \cite[Def.~5.10]{pachl2012uniform}.
	Write $\mol(G)$ for the space of all molecular measures.
	\begin{f}\cite[Thm.~6.6~and~8.18]{pachl2012uniform} \label{fact:acx_G_is_UEB_dense}
		Let $G$ be a locally compact group.
		Then $\mol(G) \cap B_{M(G)}$ is right and left UEB-dense in $B_{M(G)}$.
	\end{f}

	\section{Group and Algebra Actions with Locally Compact Groups} \label{section:group_algebra_actions}
	As we have mentioned before, we are aiming to study the relationship between the dynamics of elements of $L^{\infty}(G)$ as functions and as functionals.
	This could be done using convolutions and translations only.
	However, we find that these phenomena are best described through the language of group and algebra actions.
	In this section, we outline some needed definitions and properties regarding the group algebra and related objects.
	
	In many cases, the literature contains several sensible ways to make each definition. For example, a group $G$ can act isometrically on $L^{1}(G)$ from the left either via $(g \odot a)(h) = a(g^{-1}h)$ or ${(g\odot a) = a(hg)\Delta(g)}$.
	Here we will make a coherent choice preferable for our needs.
	We include some self-contained proofs in the appendix for the convenience of the reader.
	
	We conclude this section by recalling some known facts about the space of finite Radon measures which will also be needed later.
	
	\subsection{Basic Definitions}
	Let $G$ be a locally compact group and $\lambda$ be its left Haar measure.
	As in \cite[Thm.~15.12]{AbstractHarmonic}, we will write $\Delta\colon G \to \R^{+}$ for the modular function of $G$, namely the continuous homomorphism satisfying:
	$$
	  \forall g \in G : \lambda(A \cdot  g) = \Delta(g^{-1})\lambda(A),
	$$
	for every measurable set $A \subseteq G$.
	We write $L^{1}(G)$ and $L^{\infty}(G)$ for the spaces of absolutely integrable functions and essentially bounded functions over $G$, respectively.
	We also identify $L^{\infty}(G)$ as the dual of $L^{1}(G)$ via the map:
	$$
	\forall a\in L^{1}(G), \varphi \in L^{\infty}(G): \langle a, \varphi \rangle = \langle \varphi, a\rangle := \int_{G} a(g) \varphi(g) d \lambda(g).
	$$
	The involution $\widetilde{\cdot}\colon L^{1}(G) \to L^{1}(G)$ is defined as:
	$$
	\widetilde{a}(g) := a(g^{-1}) \Delta(g^{-1}).
	$$
	We also have:
	$$
	\breve{a}(g) := a(g^{-1}).
	$$
	\begin{f} \label{fact:inverse_integration} \cite[Thm.~15.14]{AbstractHarmonic}
		For every $\varphi \in L^{1}(G)$ we have:
		$$
		\int_{G} \widetilde{\varphi}(g) d\lambda(g) = 
		\int_{G} \varphi(g^{-1}) \Delta(g^{-1}) d\lambda(g) = 
		\int_{G} \varphi(g) d\lambda(g).
		$$
	\end{f}
	
	\begin{theoremEnd}[end, restate, text link=The proof is given in the Appendix.]{cor}[]
		\label{cor:right_haar_translation}
		For every $a \in L^{1}(G)$ and $g \in G$, we have:
		$$
		\int_{G} a(hg) \Delta(g) d\lambda(h) = \int_{G} a(h) d\lambda(h).
		$$
	\end{theoremEnd}
	
	\begin{proofE}
		\begin{align*}
			\int_{G} a(hg) \Delta(g) d\lambda(h) & =
			\int_{G} \widetilde{a}(g^{-1}h^{-1}) \Delta(g^{-1}h^{-1}) \Delta(g) d\lambda(h) \\
			& = \int_{G} \widetilde{a}(g^{-1}h^{-1}) \Delta(h^{-1}) d\lambda(h) \\
			& = \int_{G} (g\odot\widetilde{a})(h^{-1}) \Delta(h^{-1}) d\lambda(h) \\
			& \underset{\ref{fact:inverse_integration}}{=} \int_{G} (g\odot\widetilde{a})(h) d\lambda(h) \\
			& = \int_{G} \widetilde{a}(g^{-1}h) d\lambda(h) \\
			& = \int_{G} \widetilde{a}(h) d\lambda(h) \\
			& \underset{\ref{fact:inverse_integration}}{=} \int_{G} a(h) d\lambda(h).
		\end{align*}
	\end{proofE}
	
	Suppose that $a, b$ are both measurable functions on $G$.
	Their convolution $a \star b$ is defined as:
	$$
	(a \star b)(g) := \int_{G} a(h) b(h^{-1}g) d\lambda(h).
	$$
	Note that this integral is not necessarily convergent.
	\begin{f}\cite[Cor.~20.14, 20.19]{AbstractHarmonic}
		\ben
			\item If $a, b \in L^{1}(G)$, then $a \star b \in L^{1}(G)$, and $\lVert a \star b \rVert_{1} \leq \lVert a \rVert_{1} \lVert b \rVert_{1}$. 
			In other words, $L^{1}(G)$ together with convolution is a Banach algebra.
			\item If $a \in L^{1}(G)$ and $\varphi \in L^{\infty}(G)$, then $a \star \varphi \in L^{\infty}(G)$ and $\varphi \star \breve{a} \in L^{\infty}(G)$.
		\een
	\end{f}

	\begin{f} \cite[32.45]{AbstractHarmonic2}\label{fact:l1_l_infinity_convolution}
		For every locally compact group we have:
		\begin{align*}
			L^{1}(G) \star L^{\infty}(G) & = L^{1}(G) \star \RUC(G) = \RUC(G),\\
			L^{\infty} \star \breve{L}^{1}(G) & = \LUC(G) \star \breve{L}^{1}(G) = \LUC(G),
		\end{align*}
		where $\breve{L}^{1}(G) = \{\breve{a} \mid a \in L^{1}(G) \}$.
	\end{f}

	\subsection{Actions and Their Properties}
	It is clear that $G$ can induce a group action on both $L^{1}(G)$ and $L^{\infty}(G)$.
	As we have just seen, $L^{1}(G)$ can also induce an action on $L^{\infty}(G)$. 
	We will soon give a list of definitions for these actions, and then show some of their properties:
	\ben
		\item Relationship to duality - Lemmas \ref{lemma:L1_G_dual_action} and \ref{lemma:G_action_on_L_inf_duality}.
		\item Associativity - Proposition \ref{prop:L1_G_action_associativity}.
		\item Relationship to convolution - Proposition \ref{ex:l_infinity_bimodule}.
	\een
	
	To avoid possible confusions, we will use different symbols depending on the action:
	\ben
		\item If $g \in G$ and $a \in L^{1}(G)$, we will write $g \odot a$ and $a \odot g$.
		\item If $g \in G$ and $\varphi \in L^{\infty}(G)$, we will write $g \cdot \varphi$ and $\varphi \cdot g$.
		\item If $a \in L^{1}(G)$ and $\varphi \in L^{\infty}(G)$, we will write $a \circledast \varphi$ and $\varphi \circledast a$. 
		This is the same notation used in Subsection \ref{subsection:banach_algebra_action_on_dual}, and we will see that the two definitions agree (Lemma \ref{lemma:L1_G_dual_action}).
	\een
	We will also exclusively use $f \star g$ to represent function convolution, whether they belong to $L^{1}(G)$ or $L^{\infty}(G)$.
	
	\begin{defin}
		Let $g, h \in G$, $a, b \in L^{1}(G)$ and $\varphi \in L^{\infty}(G)$.
		\begin{align*}
			(g\odot a)(h) &:= a(g^{-1} h)& \quad 
			(a \odot g)(h) &:= a(h g^{-1})\Delta(g^{-1})\\
			(g\cdot \varphi)(h) &:= \varphi(hg)& \quad
			(\varphi \cdot g)(h) &:= \varphi(gh) \\
			(a \circledast \varphi)(h) &:= \langle \varphi, h \odot a \rangle& \quad
			(\varphi \circledast a)(h) &:= \langle \varphi, a \odot h\rangle.
		\end{align*}
	\end{defin}

	Note that the definition of the $G$-action on $L^{\infty}(G)$ coincides with the notions of Subsection \ref{subsection:group_actions}, so using the same notation is justified.
	Conversely, the $G$-action on $L^{1}(G)$ is distinguished, and one must be cautious in cases of functions belonging to $L^{1}(G) \cap L^{\infty}(G)$.
	We will always use the appropriate symbol ($\cdot$ or $\odot$) and thus avoid possible confusions.
	\begin{theoremEnd}[end, restate, text link=The proof is given in the Appendix.]{remark}[] \label{remark:action_respect_norm}
		These actions are well-defined, meaning that: 
		$$
		  g \odot a, a \odot g \in L^{1}(G) \text{ and } a \circledast \varphi, \varphi \circledast a, g \cdot \varphi, \varphi \cdot g \in L^{\infty}(G).
		$$
		In fact, they respect the norms of $L^{1}(G)$ and $L^{\infty}(G)$ in the following sense:
		\begin{align*}
			\lVert \varphi \rVert_{\infty} & = \lVert g \cdot \varphi \rVert_{\infty} = \lVert \varphi \cdot g \rVert_{\infty}\\
			\lVert a \rVert_{1} & = 
			\lVert g \odot a \rVert_{1} = 
			\lVert a \odot g \rVert_{1}\\
			\lVert a \rVert_{1}\lVert \varphi \rVert_{\infty} & \geq \lVert a \circledast \varphi  \rVert_{\infty}, \lVert \varphi \circledast a \rVert_{\infty}.
		\end{align*}
	\end{theoremEnd}
	\begin{proofE}
		The first equality is easy to prove.
		As for the second:
		\begin{align*}
			\lVert g \odot a \rVert_{1} & = 
			\int_{G} \lvert(g \odot a)(h)\rvert d\lambda(h)\\
			& = \int_{G} \lvert a(g^{-1}h)\rvert d\lambda(h)\\
			& = \int_{G} \lvert a(h)\rvert d\lambda(h)\\
			& = \lVert a \rVert_{1}\\
			\lVert a \odot g \rVert_{1} & = 
			\int_{G} \lvert(a \odot g)(h)\rvert d\lambda(h)\\
			& = \int_{G} \lvert a(hg^{-1})\Delta(g^{-1})\rvert d\lambda(h)\\
			& = \int_{G} \lvert a(hg^{-1})\rvert \Delta(g^{-1}) d\lambda(h)\\
			& \underset{\ref{cor:right_haar_translation}}{=} \int_{G} \lvert a(h)\rvert d\lambda(h)\\
			& = \lVert a \rVert_{1}.
		\end{align*}
		As a consequence, the final inequality is also easy.
	\end{proofE}

	\begin{lemmaE}[][end, restate, text link=The proof is given in the Appendix.] \label{lemma:L1_G_dual_action}
		For every $\varphi \in L^{\infty}(G)$ and $a, b \in L^{1}(G)$, we have:
		\begin{align*}
			\langle a \circledast \varphi, b \rangle & = \langle \varphi, b \star a \rangle,\\
			\langle \varphi \circledast a, b\rangle & = \langle \varphi, a \star b\rangle.\\
		\end{align*}
		In other words, $\circledast$ is exactly the action of the Banach algebra $L^{1}(G)$ on its dual $L^{\infty}(G)$ defined in Subsection \ref{subsection:banach_algebra_action_on_dual}.
	\end{lemmaE}
	\begin{proofE}
		In the following progression, the Fubini-Tonelli theorem \cite[Thm.~13.10]{AbstractHarmonic} is referred to as ``F.T.". 
		A justification will appear later.
		\begin{align*}
			\langle a \circledast \varphi, b \rangle & = 
			\int_{G} (a \circledast \varphi)(g) b(g) d\lambda(g)\\
			& = \int_{G} \langle \varphi, g \odot a\rangle b(g) d\lambda(g)\\
			& = \int_{G} b(g) \int_{G} \varphi(h) (g \odot a)(h) d\lambda(h) d\lambda(g)\\
			& = \int_{G}  b(g) \int_{G} \varphi(h) a(g^{-1}h) d\lambda(h) d\lambda(g)\\
			& \underset{F.T.}{=} \int_{G} \varphi(h) \int_{G} b(g) a(g^{-1}h) d\lambda(g) d\lambda(h)\\
			& = \int_{G} \varphi(h) (b \star a)(h) d\lambda(h)\\
			& = \langle \varphi, b \star a \rangle\\
			\langle \varphi \circledast a, b\rangle & =
			\int_{G} (\varphi \circledast a)(g) b(g) d\lambda(g)\\
			& = \int_{G} \langle \varphi, a \odot g \rangle b(g) d\lambda(g)\\
			& = \int_{G} b(g) \int_{G} \varphi(h) (a \odot g)(h) d\lambda(h) d\lambda(g)\\
			& = \int_{G} b(g) \int_{G} \varphi(h) a(hg^{-1})\Delta(g^{-1}) d\lambda(h) d\lambda(g)\\
			& \underset{F.T.}{=} \int_{G} \varphi(h) \int_{G} b(g) a(hg^{-1})\Delta(g^{-1}) d\lambda(g) d\lambda(h)\\
			& \underset{\ref{fact:inverse_integration}}{=} \int_{G} \varphi(h) \int_{G} a(hg) b(g^{-1}) d\lambda(g) d\lambda(h)\\
			& = \int_{G} \varphi(h) \int_{G} a(g) b(g^{-1}h) d\lambda(g) d\lambda(h)\\
			& = \int_{G} \varphi(h) (a \star b)(h) d\lambda(h)\\
			& = \langle \varphi, a \star b \rangle.
		\end{align*}
		To justify the application of the Fubini-Tonelli theorem, we need to show that any of the integrals are absolutely convergent.
		Indeed, using Remark \ref{remark:action_respect_norm}:
		\begin{align*}
			\int_{G}\int_{G} \lvert b(g) \varphi(h) (g \odot a)(h)\rvert d\lambda(h) d\lambda(g) & \leq
			\lVert \varphi \rVert_{\infty} \int_{G} \int_{G} \lvert b(g) (g\odot a)(h) \rvert d\lambda(h) d\lambda(g)\\
			& = \lVert \varphi \rVert_{\infty} \int_{G} \lvert b(g) \rvert \lVert g\odot a\rVert_{1} d\lambda(g)\\
			& = \lVert \varphi \rVert_{\infty} \int_{G} \lvert b(g) \rvert \lVert a\rVert_{1} d\lambda(g)\\
			& = \lVert \varphi \rVert_{\infty} \lVert a\rVert_{1} \int_{G} \lvert b(g) \rvert d\lambda(g)\\
			& = \lVert \varphi \rVert_{\infty} \lVert a\rVert_{1} \lVert b \rVert_{1} < \infty.
		\end{align*}
		A very similar argument justifies the second application of the Fubini-Tonelli theorem.
	\end{proofE}
	\begin{lemmaE} [][end, restate, text link=The proof is given in the Appendix.] \label{lemma:G_action_on_L_inf_duality}
		For every $g \in G, a\in L^{1}(G), \varphi \in L^{\infty}(G)$:
		\begin{align*}			
			\langle g \cdot \varphi, a\rangle & = \langle \varphi, a \odot g\rangle := (\varphi \circledast a)(g)\\
			\langle \varphi \cdot g, a \rangle & = \langle \varphi, g\odot a\rangle := (a \circledast \varphi)(g).
		\end{align*}
	\end{lemmaE}
	\begin{proofE}
		\begin{align*}
			\langle g \cdot \varphi, a\rangle & = 
			\int_{G} (g \cdot \varphi)(h) a(h) d\lambda(h)\\
			& = \int_{G} \varphi(hg) a(h) d\lambda(h)\\
			& = \int_{G} \varphi(hg) \Delta(g) \Delta(g^{-1})a(hgg^{-1}) d\lambda(h)\\
			& \underset{\ref{cor:right_haar_translation}}{=} \int_{G} \varphi(h) a(hg^{-1}) \Delta(g^{-1}) d\lambda(h)\\
			& = \int_{G} \varphi(h) (a \odot g)(h) d\lambda(h)\\
			& = \langle \varphi, a \odot g\rangle,\\
			\langle \varphi \cdot g, a \rangle & = 
			\int_{G} (\varphi \cdot g)(h) a(h)d \lambda(h)\\
			& = \int_{G} \varphi(g h) a(h)d \lambda(h)\\
			& = \int_{G} \varphi(h) a(g^{-1} h)d \lambda(h)\\
			& = \int_{G} \varphi(h) (g \odot a)(h)d \lambda(h)\\
			& = \langle \varphi, g \odot a \rangle.
		\end{align*}
	\end{proofE}
	
	\begin{theoremEnd}[end, restate, text link=The proof is given in the Appendix.]{prop}[]
		\label{prop:L1_G_action_associativity}
		Let $g, h \in G$, $a, b \in L^{1}(G)$ and $\varphi \in L^{\infty}(G)$.
		\begin{align}
			g \odot (a \odot h) & = (g \odot a) \odot h \label{prop:L1_G_action_associativity:L1_G_bimodule}\\
			g \cdot (\varphi \cdot h) & = (g \cdot \varphi) \cdot h \label{prop:L1_G_action_associativity:L_inf_G_bimodule}\\
			g \odot (a \star b) & = (g \odot a) \star b \label{prop:L1_G_action_associativity:G_L1_L1}\\
			a \star (b \odot g) & = (a \star b) \odot g \label{prop:L1_G_action_associativity:L1_L1_G}\\
			a \star (g \odot b) & = (a \odot g) \star b \label{prop:associativity_of_locally_compact_groups:L1_group_L1}\\
			a \circledast (\varphi \circledast b) & = (a \circledast \varphi) \circledast b \label{prop:L1_G_action_associativity:L_inf_L1_bimodul}\\
			g \cdot (\varphi \circledast a) & = (g \cdot \varphi) \circledast a \label{prop:L1_G_action_associativity:G_L_inf_L1}\\
			a \circledast (\varphi \cdot g) & = (a \circledast \varphi) \cdot g \label{prop:L1_G_action_associativity:L1_L_inf_G}\\
			g \cdot (a \circledast \varphi) & = (g \odot a) \circledast \varphi \label{prop:associativity_of_locally_compact_groups:group_algebra_functional}\\
			\varphi \circledast (a \odot g) & = (\varphi \circledast a) \cdot g \label{prop:associativity_of_locally_compact_groups:functional_algebra_group}\\
			a \circledast (g \cdot \varphi) & = (a \odot g) \circledast \varphi \label{prop:associativity_of_locally_compact_groups:L1_group_L_inf}\\
			\varphi \circledast (g \odot a) & = (\varphi \cdot g) \circledast a. \label{prop:associativity_of_locally_compact_groups:L_inf_froup_L1}
		\end{align}
	\end{theoremEnd}
	\begin{proofE}
		Suppose that $x \in G$.
		\begin{align*}
			\tag{\ref{prop:L1_G_action_associativity:L1_G_bimodule}}
			(g \odot (a \odot h))(x) & =
			(a \odot h)(g^{-1}x) \\
			& = a(g^{-1}x h^{-1}) \Delta(h)\\
			& = (g \odot a)(x h^{-1}) \Delta(h)\\
			& = ((g \odot a) \odot h)(x),
		\end{align*}
		\begin{align*}
			\tag{\ref{prop:L1_G_action_associativity:L_inf_G_bimodule}}
			(g \cdot (\varphi \cdot h))(x) & = (\varphi \cdot h)(x g)\\
			& = \varphi(hxg)\\
			& = (g \cdot \varphi)(hx)\\
			& = ((g \cdot \varphi) \cdot h)(x),
		\end{align*}
		\begin{align*}
			 \tag{\ref{prop:L1_G_action_associativity:G_L1_L1}}
			(g \odot (a \star b))(h) & = (a \star b)(g^{-1}h) \\
			& = \int_{G} a(x) b(x^{-1}g^{-1} h) d\lambda(x)\\
			& = \int_{G} a(g^{-1} g x) b((g x)^{-1} h) d\lambda(x)\\
			& = \int_{G} a(g^{-1} x) b(x^{-1} h) d\lambda(x)\\
			& = \int_{G} (g \odot a)(x) b(x^{-1} h) d\lambda(x)\\
			& = ((g \odot a) \star b)(h),
		\end{align*}
		\begin{align*}
			\tag{\ref{prop:L1_G_action_associativity:L1_L1_G}}
			(a \star (b \odot g))(h) & = 
			\int_{G} a(x) (b\odot g)(x^{-1} h) d\lambda(x)\\
			& = \int_{G} a(x) b(x^{-1} h g^{-1}) \Delta(g^{-1}) d\lambda(x)\\
			& = \Delta(g^{-1}) \int_{G} a(x) b(x^{-1} (h g^{-1})) d\lambda(x)\\
			& = (a \star b)(hg^{-1}) \Delta(g^{-1})\\
			& = ((a \star b) \odot g)(h),
		\end{align*}
		\begin{align*}
			\tag{\ref{prop:associativity_of_locally_compact_groups:L1_group_L1}}
			(a \star (g \odot b))(h) & = 
			\int_{G} a(x) (g \odot b)(x^{-1}h) d\lambda(x)\\
			& = \int_{G} a(x) b(g^{-1}x^{-1}h) d\lambda(x)\\
			& \underset{\ref{cor:right_haar_translation}}{=} \int_{G} a(xg^{-1})\Delta(g^{-1}) b(x^{-1}h) d\lambda(x)\\
			& = \int_{G} (a\odot g)(x) b(x^{-1}h) d\lambda(x)\\
			& = ((a \odot g) \star b)(h),
		\end{align*}
		\begin{align*}
			\tag{\ref{prop:L1_G_action_associativity:L_inf_L1_bimodul}}
			(a \circledast (\varphi \circledast b))(g) & = 
			\langle \varphi \circledast b, g\odot a\rangle\\
			& \underset{\ref{lemma:L1_G_dual_action}}{=} \langle \varphi, b \star (g\odot a) \rangle\\
			& \underset{\ref{prop:associativity_of_locally_compact_groups:L1_group_L1}}{=} \langle \varphi, (b \odot g) \star a\rangle\\
			& \underset{\ref{lemma:L1_G_dual_action}}{=} \langle a \circledast \varphi, b\odot g\rangle\\
			& = ((a \circledast \varphi) \circledast b)(g),
		\end{align*}
		\begin{align*}
			\tag{\ref{prop:L1_G_action_associativity:G_L_inf_L1}}
			(g \cdot (\varphi \circledast a))(h) & = 
			(\varphi \circledast a)(hg)\\
			& = \langle \varphi, a \odot (hg)\rangle\\
			& = \langle \varphi, (a \odot h) \odot g)\rangle\\
			& \underset{\ref{lemma:G_action_on_L_inf_duality}}{=} \langle g\cdot \varphi, a\odot h\rangle\\
			& = ((g\cdot \varphi) \circledast a)(h),
		\end{align*}
		\begin{align*}
			\tag{\ref{prop:L1_G_action_associativity:L1_L_inf_G}}
			(a \circledast (\varphi \cdot g))(h) & = 
			\langle \varphi \cdot g, h \odot a \rangle\\
			& \underset{\ref{lemma:G_action_on_L_inf_duality}}{=} \langle \varphi, g\odot (h\odot a)\rangle \\
			& = \langle \varphi, (gh) \odot a \rangle\\
			& = (a \circledast \varphi)(gh)\\
			& = ((a\circledast \varphi) \cdot g)(h),
		\end{align*}
		\begin{align*}
			\tag{\ref{prop:associativity_of_locally_compact_groups:group_algebra_functional}}
			(g \cdot (a \circledast \varphi))(h) & = (a \circledast \varphi)(hg)\\
			& = \langle \varphi, (hg) \odot a \rangle\\
			& = \langle \varphi, h \odot (g \odot a) \rangle\\
			& = ((g\odot a) \circledast \varphi)(h),
		\end{align*}
		\begin{align*}
			\tag{\ref{prop:associativity_of_locally_compact_groups:functional_algebra_group}}
			(\varphi \circledast (a \odot g))(h) & = 
			\langle \varphi, (a\odot g) \odot h\rangle\\
			& = \langle \varphi, a\odot (gh) \rangle\\
			& = (\varphi \circledast a)(gh)\\
			& = ((\varphi \circledast a)\cdot g)(h),
		\end{align*}
		\begin{align*}
			\tag{\ref{prop:associativity_of_locally_compact_groups:L1_group_L_inf}}
			(a \circledast (g \cdot \varphi))(h) & = \langle g\cdot \varphi, h \odot a\rangle\\
			& \underset{\ref{lemma:G_action_on_L_inf_duality}}{=} \langle \varphi, (h \odot a) \odot g\rangle\\
			& \underset{\ref{prop:L1_G_action_associativity:L1_G_bimodule}}{=}  \langle \varphi, h \odot (a\odot g)\rangle\\
			& = ((a \odot g) \circledast \varphi)(h),
		\end{align*}
		\begin{align*}
			\tag{\ref{prop:associativity_of_locally_compact_groups:L_inf_froup_L1}}
			(\varphi \circledast (g \odot a))(h) & = 
			\langle \varphi, (g\odot a) \odot h\rangle\\
			& \underset{\ref{prop:L1_G_action_associativity:L1_G_bimodule}}{=} \langle \varphi, g \odot (a \odot h) \rangle\\
			& \underset{\ref{lemma:G_action_on_L_inf_duality}}{=} \langle \varphi\cdot g, a\odot h\rangle\\
			& = ((\varphi \cdot g) \circledast a)(h).
		\end{align*}
	\end{proofE}


	\begin{theoremEnd}[end, restate, text link=The proof is given in the Appendix.]{prop}[]\cite[Prop.~3.6]{Yassin} \label{ex:l_infinity_bimodule}
		For every $a \in L^{1}(G)$ and $\varphi \in L^{\infty}(G)$, we have
		\begin{align*}
			a \circledast \varphi & = \varphi\star\breve{a},\\
			\varphi \circledast a & = \widetilde{a} \star \varphi.
		\end{align*}
	\end{theoremEnd}
	\begin{proofE}
		\begin{align*}
			(a \circledast \varphi)(g) & := \langle \varphi, g \odot a\rangle\\
			& = \int_{G} \varphi(h) (g\odot a)(h) d\lambda(h)\\
			& = \int_{G} \varphi(h) a(g^{-1} h) d\lambda(h)\\
			& = \int_{G} \varphi(h) \breve{a}(h^{-1}g) d\lambda(h)\\
			& = (\varphi \star \breve{a})(g),\\
			(\varphi \circledast a)(g) & := \langle \varphi, a \odot g \rangle\\
			& = \int_{G} \varphi(h) (a\odot g)(h) d\lambda(h)\\
			& = \int_{G} \varphi(h) a(hg^{-1})\Delta(g^{-1}) d\lambda(h)\\
			& = \int_{G} \varphi(h) a(hg^{-1})\Delta(hg^{-1})\Delta(h^{-1}) d\lambda(h)\\
			& = \int_{G} \varphi(h) \widetilde{a}(gh^{-1})\Delta(h^{-1}) d\lambda(h)\\
			& \underset{\ref{fact:inverse_integration}}{=} \int_{G} \varphi(h^{-1}) \widetilde{a}(gh) d\lambda(h)\\
			& = \int_{G} \widetilde{a}(h) \varphi(h^{-1}g) d\lambda(h)\\
			& = (\widetilde{a} \star \varphi)(g).
		\end{align*}
	\end{proofE}

	\begin{f} \label{fact:RUC_is_essential}
		For every locally compact $G$, $\RUC(G)$ (resp. $\LUC(G)$) is an essential right (resp. left) $L^{1}(G)$-module.
	\end{f}
	\begin{proof}
		Apply Fact \ref{fact:l1_l_infinity_convolution} and Proposition \ref{ex:l_infinity_bimodule}.
	\end{proof}

	\subsection{The Space of Finite Radon Measures}
	Let $M(G)$ be the space of finite Radon measures on $G$. 
	We will consider it as a subspace of the dual of $C_{b}(G)$ via the map
	$$
	  \forall \varphi \in C_{b}(G), \mu \in M(G): \langle \varphi, \mu \rangle  = \langle \mu, \varphi\rangle := \int_{G}\varphi(g)d \mu(g).
	$$
	We will identify $L^{1}(G)$ and $G$ as subsets of $M(G)$ via the natural maps $\Phi\colon L^{1}(G) \to M(G), {\Psi\colon G\to M(G)}$:
	\begin{align*}
		\forall a \in L^{1}(G): (\Phi(a))(A) & := \int_{A} a(h) d\lambda(h)\\
		\forall g \in G: (\Psi(g))(A) & := \begin{cases}
			1 & g \in A\\
			0 & g \notin A
		\end{cases},\\
	\end{align*}
	for measurable $A \subseteq G$.
	Note that $\acx\Psi(G) = \mol(G) \cap B_{M(G)}$.

	\begin{theoremEnd}[end, restate, text link=The proof is given in the Appendix.]{lemma}[] \label{lemma:duality_of_measures_embedding}
		For every $\varphi \in C_{b}(G), g \in G$ and $a \in L^{1}(G)$, we have:
		\begin{align*}
			\langle \varphi, \Psi(g)\rangle & = \varphi(g)\\
			\langle \varphi, \Phi(a)\rangle & =\langle \varphi, a\rangle.
		\end{align*}
	\end{theoremEnd}
	\begin{proofE}
		Indeed:
		\begin{align*}
			\langle \varphi, \Psi(g)\rangle & =
			\int_{G} \varphi(h) d(\Psi(g))(h)\\
			& = \varphi(g),
		\end{align*}
		\begin{align*}
			\langle \varphi, \Phi(a)\rangle & = 
			\int_{G} \varphi(h) (d\Phi(a))(h)\\
			& = \int_{G} \varphi(h) a(h)dh\\
			& = \langle \varphi, a\rangle.
		\end{align*}
	\end{proofE}

	\section{Relation Between Tameness and Co-Tameness} \label{section:tameness_and_cotameness}
	Let $V$ be a Banach space and $M \subseteq V^{*}$ be a weak-star compact subset.
	Here we will study the relation between $M$ being tame and co-tame.
%
%
%

	\begin{lemma} \label{lemma:independent_sequence_inversion}
		Let $K$ be a compact space and $F = \{f_{n}\}_{n \in \N} \subseteq C(K)$ be an independent sequence over $K$.
		Write $r\colon K \to \R^{C(K)}$ for the evaluation map.
		Then there exists a sequence $\{x_{n}\}_{n \in \N} \subseteq K$ such that $\{r(x_{n})\}_{n \in \N}$ is independent over $F$.
	\end{lemma}
	\begin{proof}
		Let $a < b \in \R$ be the bounds of independence of $F$, and let $Z \subseteq \{0, 1\}^{\N}$ be the subset of finitely supported sequences.
		It is known that $Z$ is countable, so let $Z = \{\eta^{(n)}\}_{n \in \N}$ be an enumeration of $Z$.
		Write $\nu^{(m)} := \{\eta^{(n)}_{m}\}_{n \in \N} \in \{0, 1\}^{\N}$ for $m \in \N$.
		In virtue of Fact \ref{f:independence_over_compact_sets}, there exists some $x_{m} \in K$ such that 
		$$
		f_{n}(x_{m}) \in \begin{cases}
			(-\infty, a] & \nu^{(m)}_{n} = 1 \\
			[b, \infty) & \nu^{(m)}_{n} = 0 \\
		\end{cases}.
		$$
		We claim that $\{x_{m}\}_{m \in \N}$ is independent over $F$ with boundaries $a < b$.
		Indeed, suppose that $N, M \subseteq \N$ are finite and disjoint.
		We need to find $f \in F$ such that
		$$
		\forall k \in N, m \in M: f(x_{k}) \leq a \text{ and } f(x_{m}) \geq b.
		$$
		Consider the characteristic function $1_{N} \in Z$. 
		There exists some $n_{0} \in \N$ such that $\eta^{(n_{0})} = 1_{N}$.
		Note that for every $k \in N, m \in M$, we have $\nu_{n_{0}}^{(k)} := \eta_{k}^{(n_{0})} = 1_{N}(k) = 1$ and similarly, $\nu_{n_{0}}^{(m)} = 0$.
		Thus
		$$
		\forall k \in N, m \in M: f_{n_{0}}(x_{k}) \leq a \text{ and } f_{n_{0}}(x_{m}) \geq b.
		$$
	\end{proof}
	
	\begin{remark}
		The intuition for the previous proof comes from the case of the standard basis ${\{e_{n}\}_{n \in \N} \subseteq l^{1}}$ and the dual ball $B_{l^{\infty}}$ \cite[p.~211, exercise~1]{Diestel}.
		We can construct a sequence ${\{\varphi^{(n)}\}_{n \in \N} \subseteq B_{l^{\infty}}}$ such that for every choice $c_{1}, \dotsc, c_{N} \in \R$, there is at least one coordinate $k \in \N$ such that:
		$$
		  \forall 1 \leq n \leq N: \varphi^{(n)}(e_{k}) = \sign(c_{n}),
		$$
		and therefore
		$$
		\left\lVert \sum_{n=1}^{N} c_{n} \varphi^{(n)} \right\rVert_{\infty} \geq\left\lvert \sum_{n=1}^{N} c_{n} \varphi^{(n)}(e_{k}) \right\rvert = \left\lvert\sum_{i=1}^{m} c_{i} \sign c_{i} \right\rvert = \sum_{i=1}^{m} \lvert c_{i} \rvert.
		$$
		In particular, let $\{\nu^{(m)}\}_{m \in \N}$ be a lexicographic enumeration of all the finite sequences of words of $\{-1, 1\}$.
		We can define $\varphi_{n}(e_{m}) := \nu^{(m)}_{n}$.
%
	\end{remark}
%
%
	\begin{thm} \label{thm:tame_inversion}
		Let $V$ be a Banach space and let $M \subseteq V^{*}$ be a weak-star compact subset.
		If $M$ is tame as a bounded subset of $V^{*}$ (i.e., tame as a family of functions over $B_{V^{**}}$), then it is co-tame.
	\end{thm}
	\begin{proof}
		By contradiction, assume that $M$ is not co-tame.
		We can therefore find a sequence ${A = \{x_{n}\}_{n \in \N}}$ in $B_{V}$ such that $A$ is independent over $M$.
		Applying Lemma \ref{lemma:independent_sequence_inversion}, we can find a sequence $S \subseteq M$ which is independent over $A$.
		Write $J\colon V \to V^{**}$ for the canonical embedding.
		It is easy to see that $S \subseteq M$ is also independent over $J(A) \subseteq B_{V^{**}}$.
		As a consequence, $M$ is not tame, a contradiction.
	\end{proof}
	Note that the converse of Theorem \ref{thm:tame_inversion} need not be true.
	Consider $V := c_{0}$ and $M := B_{l^{1}}$ (recall that $c_{0}^{*} = l^{1}$).
	It is easy to see that $M$ is co-tame since $V$ itself is Rosenthal (in fact Asplund having a separable dual), but $M$ is clearly not tame.
	The following example shows that it needn't be the case for $L^{\infty}$ spaces.
	
	\begin{ex} \label{ex:cotame_non_tame_in_l_infinity}
		There exists a bounded sequence $M$ in $l^{\infty}$ which is:
		\ben
		\item weak-star compact;
		\item independent over the standard basis $\{e_{n}\}_{n \in \N}$;
		\item co-tame.
		\een				
		Namely, $M := \{0\} \cap \{\{\alpha^{(m)}_{n} \}_{n \in \N}\}_{m \in \N} \subseteq l^{\infty}$ and:
		$$
		\alpha^{(m)}_{n} := \text{ the coefficient of $2^{m}$ in the dyadic expansion of $n$}
		$$	
		for $0 \leq n, m \in \N$.	
		\begin{center}
			\begin{tabular}{|l|l|l|l|l|}
				$\alpha^{(m)}_{n}$ & $m=0$ & $m=1$ & $m=2$ & $\dots$ \\
				\hline
				$n=0$                & 0   & 0   & 0 & $\dots$  \\
				$n=1$                & 1   & 0   & 0 & $\dots$   \\
				$n=2$                & 0   & 1   & 0 & $\dots$   \\
				$n=3$                & 1   & 1   & 0 & $\dots$   \\
				$n=4$                & 0   & 0   & 1 & $\dots$   \\
				$n=5$                & 1   & 0   & 1 & $\dots$   \\
				$n=6$                & 0   & 1   & 1 & $\dots$   \\
				$n=7$                & 1   & 1   & 1 & $\dots$   \\
				$\vdots$  		   & $\vdots$ & $\vdots$ & $\vdots$ & $\ddots$ 
			\end{tabular}
		\end{center}
	\end{ex}
	\begin{proof}~\\
		\ben
		\item It is easy to see that $M$ is bounded and it is known that the weak-star topology coincides with the topology of coordinatewise convergence on bounded subsets of $l^{\infty}$. 
		Moreover, $\lim\limits_{m \in \N} \alpha_{n}^{(m)} = 0$ for every $n \in \N$. Therefore
		$$
		  \lim_{m \in \N} \alpha^{(m)} = 0 \in M.
		$$
		In other words, $M$ is a weak-star converging sequence together with its limit, hence a weak-star compact subset.
		\item Indeed, suppose that $N, M \subseteq \N$ are disjoint and finite. 
		We will find $n_{0} \in \N$ such that
		$$
		  \langle \alpha^{(m)}, e_{n_{0}}\rangle = \begin{cases}
		  	1 & m \in N\\
		  	0 & m \in M
		  \end{cases}.
		$$
		Write
		$$
		n_{0} := \sum_{k \in N} 2^{k}.
		$$
		Now, by definition
		\begin{align*}
			\langle \alpha^{(m)}, e_{n_{0}}\rangle & =
			\alpha^{(m)}_{n_{0}} =
			\begin{cases}
				1 & m \in N \\
				0 & \text{else}
			\end{cases},
		\end{align*}
		as required.
		\item By contradiction, assume that $M$ is not co-tame.
		As a consequence, there is some subset $A \subseteq l^{1}$ which is independent over $M$.
		However, since $M$ was shown to be weak-star compact, Fact \ref{f:independence_over_compact_sets} would imply that $M$ has cardinality equal or larger than the continuum, which is clearly false as $M$ is countable.
		\een
	\end{proof}

	\section{Small Functions and Functionals}\label{section:small_functions_and_functionals}

	\subsection{Small Functionals on Banach Algebras}
	We use bornologies to generalize WAP, Asplund and tame functionals (defined in Subsection \ref{subsection:classes_of_functionals}).

	\begin{defin} \label{defin:B_small_functional}
		Let $\A$ be a Banach algebra and $\mathcal{B}$ be a vector bornology on $\ruc(\A)$.
		A functional $\varphi \in \ruc(\A)$ is said to be \emph{right $\mathcal{B}$-small} if $\mathfrak{L}_{\varphi}(B_{\A}) \in \mathcal{B}$.
		Recall that $\mathfrak{L}_{\varphi}(B_{\A}) = \varphi \circledast B_{\A}\subseteq \A^{*}$.
		We will also write $\absmall[\mathcal{B}]{\ruc(\A^{*})}$ for the set of all right $\mathcal{B}$-small functionals.
		Here ``a" stands for ``algebra".
	\end{defin}
	\begin{remark}
		Analogous definitions could be made for the left action.
	\end{remark}
	We will write $\absmall[\BDLP]{\RUC(G)}, \absmall[\BNP]{\RUC(G)}, \absmall[\BTame]{\RUC(G)}$ for the right small functionals induced by the weakly relatively compact, Asplund and tame bornologies on $L^{1}(G)$ (Fact \ref{fact:bornologies}).
	\begin{remark} \label{remark:equivalence_functional_tameness}
		A functional is WAP in the sense of \cite{Ulger}, if and only if it is right $\BDLP$-small.
		A functional is tame/Asplund in the sense of \cite{TameFunc}, if and only if it is right $\BTame$-small/$\BNP$-small.
		In other words:
		$$
		\wapfunctional(L^{1}(G)) = \absmall[\BDLP]{\RUC(G)},
		$$
		$$
		\aspfunctional(L^{1}(G)) = \absmall[\BNP]{\RUC(G)},
		$$
		$$
		\tamefunctional(L^{1}(G)) = \absmall[\BTame]{\RUC(G)}.
		$$
	\end{remark}
%
	
%

	\begin{prop}
		The set $\absmall[\mathcal{B}]{\ruc(\A^{*})}$ of all right $\mathcal{B}$-small functionals on $\A$ is a linear subspace of $\A^{*}$.
	\end{prop}
	\begin{proof}
		First, suppose that $\varphi, \psi \in \absmall[\mathcal{B}]{\ruc(\A^{*})}$ and $\alpha \in \R$.
		It is easy to see that:
		\begin{align*}
			\mathfrak{L}_{\varphi + \alpha \psi}(B_{\A}) & =
			(\mathfrak{L}_{\varphi} + \alpha \mathfrak{L}_{\psi})(B_{\A}) \\
			& \subseteq \mathfrak{L}_{\varphi}(B_{\A}) + \alpha \mathfrak{L}_{\psi}(B_{\A}).
		\end{align*}
		By definition, $\mathfrak{L}_{\varphi}(B_{\A})$ and $\mathfrak{L}_{\psi}(B_{\A})$ are right $\mathcal{B}$-small.
		Since $\mathcal{B}$ is a vector bornology, their linear combination is also right $\mathcal{B}$-small, as required.
	\end{proof}

	
	\begin{ex}
		Consider the Banach algebra $\A := l^{1}$ together with pointwise multiplication.
		Then every functional $\varphi \in \A^{*}$ is right $\BDLP$-small.
	\end{ex}
	\begin{proof}
		Let $\varphi \in l^{\infty} = \A^{*}$.
		We will denote multiplication in $\A$ by $a \star b$ for $a, b \in l^{1}$, and use $\circledast$ for the module operation on $l^{\infty}$.
		We will show that $\varphi \circledast B_{l^{1}}$ is weakly relatively compact.		
		It is easy to see that for every $a\in l^{1}$ we have:
		$$
		  \varphi \circledast a = \varphi \star a,
		$$
		where in the right hand side we consider element-wise multiplication.
		As a consequence, 
		$$
		  \varphi \circledast B_{l^{1}} \subseteq \lVert \varphi \rVert_{\infty} B_{l^{1}} \subseteq c_{0} \cap \lVert \varphi \rVert_{\infty} B_{l^{\infty}}.
		$$
		Therefore, it is enough to show that $B_{l^{1}} \subseteq l^{\infty}$ is weakly compact.
		Indeed, $B_{l^{1}}$ is closed in the topology of  coordinatewise convergence, and therefore also weak-star closed in $l^{\infty}$.
		In virtue of the Banach–Alaoglu theorem, $B_{l^{1}}$ is weak-star compact.
		In other words, it is compact with respect to the weak topology induced from $l^{1} = c_{0}^{*}$.
		However, $B_{l^{1}} \subseteq c_{0}$, so this topology coincides with the weak topology of $c_{0} \subseteq l^{\infty}$, proving the desired result.
	\end{proof}
	
	\begin{q}
		Let $X$ be a Banach space and consider the Banach algebra $\A = \mathcal{L}(X, X)$ of bounded linear maps with the operator norm.
		Every $x \in X$ and $\varphi \in X^{*}$ induces a functional $\Phi_{x, \varphi} \in \A^{*}$ defined by:
		$$
		  \Phi_{x, \varphi}(T) := \varphi(T(x)).
		$$
		What are the combinations $(x, \varphi) \in X \times X^{*}$ such that $\Phi_{x, \varphi}$ is WAP/Asplund/tame (i.e., right $\BDLP$-small/$\BNP$-small/$\BTame$-small)?
	\end{q}
	\begin{remark}
		Note that $\Phi_{x, \varphi} = \Phi_{x, \varphi} \circledast id_{\A}$, so $\Phi_{x, \varphi}$ belongs to $\ruc(\A)$.
	\end{remark}

	\begin{ex}
		Consider $\A = \mathcal{L}(l^{1}, l^{1})$.
		Then $\Phi_{x, \varphi}$ is not tame for every $0 \neq x \in l^{1}, 0 \neq \varphi \in l^{\infty}$.
	\end{ex}
	\begin{proof}
		We need to show that $\Phi_{x, \varphi} \circledast B_{\A}$ has an $l^{1}$-sequence.		
		Using Lemma \ref{lemma:independent_sequence_inversion}, find 
		$$
		\{\psi_{n}\}_{n \in \N} \subseteq \{0, 1\}^{\N} \subseteq B_{l^{\infty}}
		$$ 
		which are independent over $\{e_{n}\}_{n \in \N} \subseteq l^{1}$ with bounds $0 < 1$.
		Since $0 \neq \varphi$, we can find $y_{0} \in l^{1}$ such that $\varphi(y_{0}) = 1$.
		Now, define $S_{n} \in B_{\A}$ via:
		$$
		  S_{n}(x) := \frac{1}{\lVert y_{0} \rVert}\psi_{n}(x)y_{0}.
		$$
		Using the Hahn-Banach theorem, we can find $T_{m} \in \A$ such that:
		$$
		  T_{m}(x) = \lVert x \rVert e_{m}.
		$$  
		
		We claim that $\{\Phi_{x, \varphi} \circledast S_{n} \}_{n \in \N}$ is independent over $\{T_{m}\}_{m \in \N} \subseteq B_{\A} \subseteq B_{\A^{**}}$.
		Indeed, by the construction of $\{\psi_{n}\}_{n \in \N}$, for every finite, disjoint $M, P\subseteq \N$ we have some $m(M, P) \in \N$ such that
		$$
		  \psi_{n}(e_{m(M, P)}) \in \begin{cases}
		  	[1, \infty) & n \in M \\
		  	(-\infty, 0] & n \in P
		  \end{cases}.
		$$
		Therefore:
		\begin{align*}
			(\Phi_{x, \varphi} \circledast S_{n})(T_{m(M, P)}) & :=
			\Phi_{x, \varphi}(S_{n} \circ T_{m(M, P)}) \\
			& = \varphi((S_{n} \circ T_{m(M, P)})(x)) \\
			& = \varphi(S_{n}( T_{m(M, P)}(x))) \\
			& = \varphi(S_{n}(\lVert x \rVert e_{m(M, P)}) \\
			& = \lVert x \rVert\varphi\left(\frac{1}{\lVert y_{0} \rVert}\psi_{n}(e_{m(M, P)})y_{0}\right) \\	
			& = \frac{\lVert x \rVert}{\lVert y_{0} \rVert}\psi_{n}(e_{m(M, P)}) \\
			& \in \begin{cases}
				\Big[\frac{\lVert x \rVert}{\lVert y_{0} \rVert}, \infty\Big) & n \in M \\
				(-\infty, 0] & n \in P
			\end{cases}.
		\end{align*}
	\end{proof}
	\begin{q}
		Is there any non-trivial tame functional on the Banach algebra $\mathcal{L}(l^{1}, l^{1})$?
	\end{q}

%
%
%
%
%

	\subsection{Small Functions on Groups}
	Let $G$ be a locally compact group.
	\begin{defin} \label{defin:B_small_function}
		Let $\mathcal{B}$ be a vector bornology on $\RUC(G)$.
		A function $\varphi \in \RUC(G)$ is said to be \emph{right $\mathcal{B}$-small as a function} if $\varphi \cdot G \in \mathcal{B}$.
		We will write $\fbsmall[\mathcal{B}]{\RUC(G)}$ for the set of all right $\mathcal{B}$-small functions (here ``g" stands for ``group").
	\end{defin}
	\begin{remark}
		Analogous definitions could be made for the left action.
	\end{remark}
	\begin{remark} \label{remark:tame_definition_is_the_same}
		In virtue of Definitions \ref{defin:ulger_wap_function} and \ref{defin:asplund_function}, and Proposition \ref{fact:tame_as_banach_subset_equivalence}, we have:
		$$
		  \wap(G) = \fbsmall[\BDLP]{\RUC(G)},
		$$
		$$
		\asp(G) = \fbsmall[\BNP]{\RUC(G)}.
		$$
		$$
		  \tame(G) = \fbsmall[\BTame]{\RUC(G)}.
		$$
	\end{remark}

	\begin{remark} \label{rem:left_wap_is_right_wap}
		In the case of WAP functions, it is known that a function is left WAP if and only if it is right WAP \cite[Cor~1.12]{burckel}.
	\end{remark}
	\section{Relation Between Functions and Functionals}\label{section:main}
	Let $G$ be a locally compact group.
	It is known that $\wap(G) = \wapfunctional(L^{1}(G))$ \cite[Thm.~4]{Ulger}.
	In \cite[Question~3.6]{TameFunc}, Megrelishvili asked if $\tame(G) = \tamefunctional(L^{1}(G))$.
	In this section we will show that it is true.
	


	The following is a generalization of a similar result for abelian groups which can be found in \cite[p.~236]{Kitchen1966}.
	
	\begin{prop} \label{prop:orbit_inside_ball_action}
		For every $\varphi \in \RUC(G)$ (resp. $\varphi \in \LUC(G)$), we have $\varphi\cdot G \subseteq \overline{\varphi \circledast B_{L^{1}(G)}}$ (resp. $G \cdot \varphi \subseteq \overline{B_{L^{1}(G)} \circledast \varphi}$).
	\end{prop}
	\begin{proof}
		We will only prove the case of $\RUC(G)$, the other one is similar.
		First, by Fact \ref{fact:L1G_has_approximate_identity} we can find an approximate identity $\{e_{\lambda}\}_{\lambda \in \Lambda} \subseteq B_{L^{1}(G)}$.
		Next, applying Fact \ref{fact:RUC_is_essential} and Fact \ref{fact:essential_is_module_identity} we conclude that
		$$
		\forall \psi \in \RUC(G): \lim\limits_{\lambda \in \Lambda} \psi \circledast e_{\lambda} = \psi.
		$$
		As a consequence, $\varphi \in \overline{ \{\varphi \circledast e_{\lambda} \}_{\lambda \in \Lambda}} \subseteq \overline{\varphi \circledast B_{L^{1}(G)} }$.
		Finally, in virtue of Remark \ref{remark:action_respect_norm}, for every $g \in G$ we have $B_{L^{1}(G)} \odot g = B_{L^{1}(G)}$, and therefore:
		$$
		{\varphi \cdot g \in 
		\overline{\varphi \circledast B_{L^{1}(G)}} \cdot g = 
		\overline{(\varphi \circledast B_{L^{1}(G)}) \cdot g} \underset{\ref{prop:associativity_of_locally_compact_groups:functional_algebra_group}}{=}
		\overline{\varphi \circledast (B_{L^{1}(G)} \odot g)} = 
		\overline{\varphi \circledast B_{L^{1}(G)}}}.
		$$
		The first equality is a consequence of the translation map $\varphi \mapsto \varphi \cdot g$ being an isometry.		
		This is true for every $g \in G$, proving the desired result.
	\end{proof}

	\begin{thm} \label{thm:acx_G_BL1_G_equivalence}
		For a locally compact group $G$, 
		\begin{align*}
			\forall \varphi \in \LUC(G): \overline{\acx} (G \cdot \varphi) & = \overline{B_{L^{1}(G)} \circledast \varphi}\\
			\forall \varphi \in \RUC(G): \overline{\acx} (\varphi \cdot G) & = \overline{\varphi \circledast B_{L^{1}(G)}}.
		\end{align*}
	\end{thm}
	\begin{proof}
		One inclusion was proven in Proposition \ref{prop:orbit_inside_ball_action}.
		All that is left is to show that:
		\begin{align*}
			\forall \varphi \in \LUC(G): \overline{\acx} (G \cdot \varphi) & \supseteq B_{L^{1}(G)} \circledast \varphi\\
			\forall \varphi \in \RUC(G): \overline{\acx} (\varphi \cdot G) & \supseteq \varphi \circledast B_{L^{1}(G)}.
		\end{align*}
		We will only prove the $\RUC$ case, the other is very similar.
		Indeed, suppose that $\varphi \in \RUC(G)$, $a \in B_{L^{1}(G)}$ and $\eps > 0$.
		We need to find some $\{\alpha_{k} \}_{k=1}^{n}\subseteq \R$, and $\{g_{k}\}_{k = 1}^{n} \subseteq G$ such that:
		$$
		  \left\lVert \varphi \circledast a - \sum_{k=1}^{n}\alpha_{k} (\varphi \cdot g_{k}) \right\rVert < \eps,
		$$
		and $\sum\limits_{k=1}^{n} \lvert \alpha_{k} \rvert \leq 1$.
		In virtue of Fact \ref{fact:orbits_are_UEB}, we conclude that $G \cdot \varphi$ is right UEB.
		Now, applying Fact \ref{fact:acx_G_is_UEB_dense}, we can find $\{\alpha_{k}\}_{k=1}^{n}\subseteq \R$ and $\{g_{k}\}_{k=1}^{n}$ such that:
		$$
		  \forall \psi \in G \cdot \varphi: \lvert\langle \psi, a - \mu \rangle\rvert \leq \eps
		$$
		and $\sum\limits_{k=1}^{n} \lvert \alpha_{k} \rvert \leq 1$, where $\mu = \sum\limits_{k=1}^{n} \alpha_{k} \Psi(g_{k}) \in \acx \Psi(G)$.
		In other words:
		$$
		  \forall g \in G: \lvert \langle g\cdot \varphi, a - \mu\rangle \rvert \leq \eps.
		$$
		Using Lemma \ref{lemma:G_action_on_L_inf_duality} we recall that $\langle g\cdot \varphi, a \rangle = (\varphi \circledast a)(g)$.
		Moreover:
		\begin{align*}
			\langle g\cdot \varphi, \mu\rangle & = \sum_{k=1}^{n} \alpha_{k}\langle g\cdot \varphi, \Psi(g_{k})\rangle\\
			& \underset{\ref{lemma:duality_of_measures_embedding}}{=} \sum_{k=1}^{n} \alpha_{k}(g\cdot \varphi)(g_{k})\\
			& = \sum_{k=1}^{n} \alpha_{k}\varphi(g_{k} g)\\
			& = \sum_{k=1}^{n} \alpha_{k}(\varphi \cdot g_{k} )(g).\\
		\end{align*}
		Thus:
		\begin{align*}
		\left\lVert \varphi \circledast a - \sum_{k=1}^{n}\alpha_{k} (\varphi \cdot g_{k}) \right\rVert & = 
		\sup_{g \in G} \left\lvert (\varphi \circledast a)(g) - \sum_{k=1}^{n}\alpha_{k} (\varphi \cdot g_{k})(g) \right\rvert\\
		& = \sup_{g \in G} \left\lvert \langle g \cdot \varphi, a\rangle - \langle g \cdot \varphi, \mu\rangle \right\rvert\\
		& = \sup_{g \in G} \lvert \langle g\cdot \varphi, a - \mu\rangle \rvert \leq \eps.
		\end{align*}
	
	\end{proof}

	\begin{thm} \label{thm:function_functional_equivalence}
		For every norm-saturated, convex vector bornology $\mathcal{B}$ we have:
		$$
		\absmall[\mathcal{B}]{\RUC(G)} = \fbsmall[\mathcal{B}]{\RUC(G)}.
		$$
		In particular:
		\begin{align*}
			\tamefunctional(L^{1}(G)) & = \tame(G),\\
			\aspfunctional(L^{1}(G)) & = \asp(G),\\
			\wapfunctional(L^{1}(G)) & = \wap(G).
		\end{align*}
	\end{thm}
	\begin{proof}
		First, note that by definition, $\varphi \in \absmall[\mathcal{B}]{\RUC(G)}$ if and only if $\varphi \circledast B_{L^{1}(G)} \in \mathcal{B}$.
		Since $\mathcal{B}$ is norm-saturated, this is equivalent to $\overline{\varphi \circledast B_{L^{1}(G)}}$ also belonging to $\mathcal{B}$.
		In virtue of Theorem \ref{thm:acx_G_BL1_G_equivalence}, this is the same as having $\overline{\acx}(\varphi\cdot G) \in \mathcal{B}$.
		Leveraging $\mathcal{B}$ being norm-saturated and convex, we can alternatively write $\varphi\cdot G \in \mathcal{B}$.
		By definition, $\varphi \in \fbsmall[\mathcal{B}]{\RUC(G)}$.
		This shows the equivalence of the two definitions.
		
		The implications for tame, Asplund and WAP functions are a consequence of Remarks \ref{remark:equivalence_functional_tameness}, \ref{remark:tame_definition_is_the_same} and Fact \ref{fact:bornologies}.
	\end{proof}
	\begin{remark}
		The previous theorem is also true when considering \emph{left} definitions rather than right ones.
	\end{remark}
	\begin{ex}
		The following are examples of interesting functions which are Asplund/tame on some locally compact group $G$.
		As a consequence of the previous theorem, they are also Asplund/tame as functionals over $L^{1}(G)$.
		\ben
		\item The function $\arctan \in L^{\infty}(\R)$ is Asplund but not WAP \cite[Example~7.19.3]{Me-Frag04}.
		\item The Fibonacci bisequence on $\Z$ is tame but not Asplund \cite[Example~6.1.2]{GM-MTame}.
		\een
	\end{ex}

	\section{Open Questions} \label{section:open_questions}

	\begin{q}  \label{question:non_ruc_tame_functionals}
		We constrained ourselves to $\RUC$ (resp. $\LUC$) functionals in Definition \ref{defin:B_small_functional}.
		However, it is reasonable to make the same definition for general elements of $L^{\infty}(G)$.
		For example, we can say that a (not necessarily right uniformly continuous) functional $\varphi \in L^{\infty}(G)$ is \emph{generalized tame} if $\varphi \circledast B_{L^{1}(G)}$ is a tame subset of $L^{\infty}(G)$.
		
		What can be said about such functionals?
		Are there any interesting examples which don't belong to $\RUC(G)$?
	\end{q}

	\begin{q}
		As mentioned in Remark \ref{rem:left_wap_is_right_wap}, every left WAP function is also a right WAP function.
		Is it also the case for Asplund and tame bornologies?
		
		In other words, if $\varphi \cdot G$ is an Asplund/tame subset, does it imply that $G \cdot \varphi$ is an Asplund/tame subset.
	\end{q}
	\begin{q}
		It is known that a functional $\varphi \in \A^{*}$ is WAP if and only if 
		$$
		  \forall \Theta, \varOmega \in \A^{**}: (\Theta \Box\varOmega)(\varphi) = (\Theta \Diamond \varOmega)(\varphi)
		$$
		where $\Box$ and $\Diamond$ are the first and second Arens products.
		
		Is there a similar algebraic characterization for tame functionals?
	\end{q}

%
	
	\section{Appendix}
	In this section, we give some proofs that were omitted in the main text.
	\allowdisplaybreaks
	\printProofs
	
	\bibliography{mybib}{}
	\bibliographystyle{plain}
\end{document}